\documentclass[12pt]{amsart}

\usepackage[colorlinks]{hyperref}
\usepackage{fullpage}
\usepackage{stmaryrd}
\usepackage{graphicx}
\usepackage{hyperref}
\usepackage{amsmath,amssymb,amsfonts,mathrsfs,amscd}
\usepackage{newtxtext,newtxmath}
\usepackage[all,cmtip]{xy}

\newtheorem{thm}{Theorem}[section]
\newtheorem{lem}[thm]{Lemma}
\newtheorem{cor}[thm]{Corollary}
\newtheorem{con}[thm]{Conjecture}

\newtheorem{thmA}{Theorem}

\newtheorem{corA}[thmA]{Corollary}

\theoremstyle{definition}

\theoremstyle{remark}

\numberwithin{equation}{section}

\newcommand{\NM}{\vartriangleleft}

\DeclareMathOperator{\Irr}{Irr}

\DeclareMathOperator{\N}{N}

\DeclareMathOperator{\Syl}{Syl}

\DeclareMathOperator{\GL}{GL}
\DeclareMathOperator{\SL}{SL}

\DeclareMathOperator{\IBr}{IBr}

\begin{document}

\title{Counting lifts of irreducible Brauer characters}

\author{Junwei Zhang}
\address{School of Mathematical and Statistics, Shanxi University, Taiyuan, 030006, China.}
\email{zhangjunwei@sxu.edu.cn}
\email{changxuewu@sxu.edu.cn}
\email{jinping@sxu.edu.cn}

\author{Xuewu Chang}

\author{Ping Jin*}

\thanks{*Corresponding author}

\keywords{$p$-solvable group; character; Brauer character; lift; vertex}

\date{}

\maketitle

\begin{abstract}
Let $p$ be an odd prime,
and suppose that $G$ is a $p$-solvable group and $\varphi\in\IBr(G)$
has vertex $Q$. In 2011, Cossey, Lewis and Navarro proved that
the number of lifts of $\varphi$ is at most $|Q:Q'|$ whenever $Q$ is normal in $G$.
In this paper, we present an explicit description of the set of lifts of $\varphi$ with a given vertex pair $(Q,\delta)$ under a weaker condition on $Q$, and thus generalize their result.
\end{abstract}

\section{Introduction}
Fix a prime $p$, and let $G$ be a $p$-solvable group.
If $\varphi\in\IBr(G)$ is an irreducible Brauer character, then the Fong-Swan theorem states that
there exists $\chi\in\Irr(G)$ such that the restriction $\chi^0$ of $\chi$ to the $p$-regular elements of $G$ is $\varphi$.
In this case, we say that $\chi$ is a {\bf lift} of $\varphi$, and we write $L_\varphi$ for the set of all lifts of $\varphi$.
In \cite{C2007}, Cossey proposed the following conjecture.

\begin{con}[Cossey]
Let $G$ be a $p$-solvable group, and let $\varphi\in\IBr(G)$. Then
$$|L_{\varphi}|\leq|Q:Q'|,$$
where $Q$ is a vertex for $\varphi$ (in the sense of Green).
\end{con}

Our motivation is to generalize the following result,
which is the main content of Theorem A of \cite{CLN2011}.
Note that Corollary B of that paper establishes Cossey's conjecture
in the case where $Q$ is abelian, so we will focus on the normal vertex case.

\begin{thm}[Cossey-Lewis-Navarro]\label{CLN}
Let $G$ be a $p$-solvable group with $p>2$, and $\varphi\in\IBr(G)$.
If $\varphi$ has a normal vertex $Q$, then $|L_{\varphi}|\leq|Q:Q'|$.
\end{thm}

It is known that Cossey's conjecture has a strong form for odd prime $p$.
Recall that a pair $(Q,\delta)$ is a (generalized) {\bf vertex} for $\chi\in\Irr(G)$
if there exist a subgroup $W\le G$ and some character $\gamma\in\Irr(W)$
with $\gamma^G=\chi$, and such that $Q$ is a Sylow $p$-subgroup of $W$ and $\gamma$
is {\bf factorable} with $(\gamma_p)_Q=\delta$,
where $\gamma_p$ denotes the $p$-special factor of $\gamma$
(for definitions, see \cite{N2002} or the next section).
In this situation, we refer to $(W,\gamma)$ as a {\bf nucleus} for $\chi$
(see Definition 4.4 of \cite{C2010}).
A crucial property of vertices, which is proved by Cossey and Lewis in \cite{CL2012}
(see also Lemma \ref{CL1.2} below), is that
if $\chi\in\Irr(G)$ is a lift, where $G$ is a $p$-solvable group with $p>2$,
then all vertex pairs for $\chi$ are conjugate and linear.
In particular, the definition of vertices above is consistent with
that given by Navarro in \cite{N2002} in the case where $p$ is an odd prime,
and we also use $\Irr(G|Q,\delta)$ to denote the set of all members of $\Irr(G)$ with vertex $(Q,\delta)$.

\begin{con}[Cossey-Lewis]
Let $G$ be a $p$-solvable group with $p>2$, and let $\varphi\in\IBr(G)$ with vertex $Q$.
If $\delta$ is a linear character of $Q$, then
$$|L_\varphi(Q,\delta)| \le |{\bf N}_G(Q):{\bf N}_G(Q,\delta)|,$$
where $L_{\varphi}(Q,\delta)$ denotes the set of those lifts of $\varphi$ with vertex pair $(Q,\delta)$,
and ${\bf N}_G(Q,\delta)$ is the set of elements of ${\bf N}_G(Q)$ that stabilize $\delta$.
\end{con}

To state our result, we need to introduce more notation.
We use $\IBr(G|Q)$ to denote the set of irreducible Brauer characters of $G$ with vertex $Q$.
Following Lewis \cite{L2011}, if $\varphi\in\IBr(G|Q)$ and $T$ is a subgroup of $G$,
we write $I_\varphi(T|Q)$ for the set of those members of $\IBr(T|Q)$ that induce $\varphi$.

The following is our main result, which generalizes Theorem \ref{CLN} in vertex form,
and which presents an explicit description of the set $L_\varphi(Q,\delta)$
of lifts of $\varphi$ with vertex $(Q,\delta)$.

\begin{thmA}\label{A}
Let $G$ be a $p$-solvable group for some odd prime $p$,
let $\varphi\in\IBr(G)$ have vertex $Q$, and let $\delta\in\Irr(Q)$ be a linear character.
Suppose that $N$ is a normal subgroup of $G$ such that
$Q\in\Syl_p(N)$ and $\delta$ extends to $N$.
Write $T=N{\bf N}_G(Q,\delta)$. Then the following hold.

{\rm (1)} For any $\chi\in L_\varphi(Q,\delta)$,
each irreducible constituent of $\chi_N$ is factorable.

{\rm (2)} Each Brauer character $\eta\in I_\varphi(T|Q)$
has a unique lift $\tilde\eta\in\Irr(T|Q,\delta)$.
Furthermore, $\tilde\eta^G$ lies in $L_\varphi(Q,\delta)$,
and the map $\eta\mapsto\tilde\eta^G$
defines a bijection $I_\varphi(T|Q)\to L_\varphi(Q,\delta)$.

{\rm (3)} $|L_\varphi(Q,\delta)|\le |{\bf N}_G(Q):{\bf N}_G(Q,\delta)|$.
\end{thmA}

It is clear that if $\varphi$ has a normal vertex $Q$,
then Theorem A applies (by taking $N=Q$).
Actually, as an immediate consequence of Theorem A,
we prove the following:

\begin{corA}\label{B}
Let $G$ be a $p$-solvable group with $p>2$, and
suppose that $\varphi\in\IBr(G)$ has vertex $Q$.
If $K\NM G$ is a $p'$-subgroup such that $KQ\NM G$,
then
$$|L_\varphi(Q,\delta)|\le |{\bf N}_G(Q):{\bf N}_G(Q,\delta)|$$
for each $\delta\in\Irr(Q)$.
\end{corA}

It is clear that Corollary \ref{B} applies to the group $G=\GL(2,3)$ and $p=3$,
by taking $Q\in\Syl_3(G)$ and $K=O_2(G)$, so that $KQ=\SL(2,3)$.
Note that $Q$ is not normal in $G$,
so Theorem \ref{CLN} is invalid in this simple example.

All groups considered in this paper are finite, and the notation and terminology are mostly taken from \cite{I2006,N1998}.

\section{Preliminaries}
In this section we will review some definitions and notation from \cite{I2018} and \cite{N2002};
see also \cite{JW2023}.

Let $G$ be a $\pi$-separable group for a set $\pi$ of primes.
We say that $\chi\in\Irr(G)$ is {\bf $\pi$-special} if $\chi(1)$ is a $\pi$-number
and the determinantal order $o(\theta)$ is a $\pi$-number
for every irreducible constituent $\theta$ of the restriction $\chi_S$
for every subnormal subgroup $S$ of $G$.

The following is Theorem 2.2 of \cite{I2018}.

\begin{lem}\label{prod}
Let $G$ be a $\pi$-separable group,
and suppose that $\alpha,\beta\in\Irr(G)$ are $\pi$-special and $\pi'$-special,
respectively. Then $\alpha\beta$ is irreducible.
Also, if $\alpha\beta=\alpha'\beta'$,
where $\alpha'$ is $\pi$-special and $\beta'$ is $\pi'$-special,
then $\alpha=\alpha'$ and $\beta=\beta'$.
\end{lem}

In this paper, we are particularly interested in the case where $\pi$ consists of a single prime $p$,
and in this case we say that $\chi$ is {\bf factored} if $\chi=\alpha\beta$ for some $p$-special character $\alpha$
and $p'$-special character $\beta$ of $G$.
For convenience, we often write $\chi_p=\alpha$ and $\chi_{p'}=\beta$.

\begin{lem}\label{ext}
Let $G$ be a $p$-solvable group with $P\in\Syl_p(G)$,
and let $\theta\in\Irr(P)$.
Then $\theta$ extends to $G$ if and only if $\theta(x)=\theta(y)$
whenever $x,y\in P$ are conjugate in $G$.
In this case, $\theta$ has a unique $p$-special extension to $G$.
\end{lem}
\begin{proof}
This follows by Theorem 3.6 and Corollary 3.15 of \cite{I2018} with $\pi=\{p\}$.
\end{proof}

Recall that a character is said to be $p$-rational if its values are contained in the cyclotomic field $\mathbb{Q}_n$
with $n$ not divisible by $p$, where $\mathbb{Q}_n$ is the subfield of the complex numbers generated over
the rationals by a primitive $n$th root of unity.
It is known that each $p'$-special character is $p$-rational (see Corollary 2.13 of \cite{I2018}).

\begin{lem}\label{I5.2}
Let $G$ be a $p$-solvable group with $p>2$, and
suppose that $\chi\in\Irr(G)$ is a lift.
If $\chi=\beta^G$, where $\beta$ is a $p'$-special character of some subgroup $W$
with $p'$-index in $G$, then $\chi$ is also $p'$-special.
\end{lem}
\begin{proof}
As mentioned above, we see that $\beta$ is $p$-rational, which clearly implies that $\chi$ is also $p$-rational. Then $\chi$ lies in $B_{p'}(G)$ by the comments after Theorem 5.2 of \cite{I2018} (with $p'$ in place of $\pi$).
Note that $\chi(1)=|G:W|\beta(1)$, which is a $p'$-number, and
thus $\chi$ is $p'$-special by Theorem 4.2 of \cite{I2018}.
\end{proof}

We need the following result, which appears in \cite{CL2012} as Theorem 1.2.

\begin{lem}[Cossey-Lewis]\label{CL1.2}
Let $p$ be an odd prime, and let $G$ be a $p$-solvable group.
Suppose that $\chi\in\Irr(G)$ is a lift of $\varphi\in\IBr(G)$.
Then all the vertex pairs for $\chi$ are linear and conjugate in $G$.
\end{lem}

Finally, we present an elementary result, which is the Brauer version of Lemma 2.3 of \cite{CLN2011}.

\begin{lem}\label{a}
Let $G$ be a group with $T\le G$ and let $\varphi\in\IBr(G)$.
Then the number of irreducible Brauer characters of $T$
inducing $\varphi$ is lass than or equal to $|G:T|$.
\end{lem}

\begin{proof}
Suppose that $\eta_1,\ldots,\eta_n\in\IBr(T)$ are different irreducible Brauer characters inducing $\varphi$. Then each $\eta_i(1)=\varphi(1)/|G:T|$ and
$\varphi_T=a_1\eta_1 +\cdots+a_n\eta_n +\Delta$,
where $a_i$ are some integers and $\Delta$ is either a Brauer character of $T$ or zero.
Now $\varphi(1)\geq n(\varphi(1)/|G:T|)$, and the result follows.
\end{proof}

\section{Proofs}
We begin by introducing a useful notion.
Given a character $\delta$ of $Q$, where $Q$ is a subgroup of a group $G$,
we say that $\delta$ is {\bf $G$-stable} if $\delta(x)=\delta(y)$ whenever $x,y\in Q$ are conjugate in $G$.

The following is Theorem A of \cite{WJ2023}, which is key to our purpose.

\begin{lem}\label{WJ}
Let $G$ be a $p$-solvable group, and let $Q$ be a $p$-subgroup of $G$.
Suppose that $\delta$ is a linear character of $Q$ that is $G$-stable.
Then restriction to $p$-regular elements defines a natural bijection $\Irr(G|Q,\delta)\rightarrow\IBr(G|Q)$.
\end{lem}

We need a simple property of $G$-stable characters.

\begin{lem}\label{stab}
Let $G$ be a $p$-solvable group with $N\NM G$,
and suppose that $Q\in\Syl_p(N)$ and $\delta\in\Irr(Q)$.
Assume also that $\delta$ extends to $N$ and is $H$-stable,
where $H$ is a subgroup of $G$ containing $Q$.
Then $\delta$ is $NH$-stable,
and the unique $p$-special extension of $\delta$ to $N$
is $NH$-invariant.
\end{lem}
\begin{proof}
Observe that $Q$ is a Sylow $p$-subgroup of $N\cap H$,
and so $H={\bf N}_H(Q)(N\cap H)$ by the Frattini argument.
This implies that $NH=HN={\bf N}_H(Q)N$.
To show that $\delta$ is $NH$-stable,
we need to prove that $\delta(x)=\delta(y)$ whenever $x,y\in Q$ are
conjugate in $NH$. Now we can write $y=x^{hn}=(x^h)^n$ for some $h\in {\bf N}_H(Q)$
and $n\in N$,
and thus $x^h$ and $y$ are elements of $Q$ that are conjugate in $N$.
It follows by Lemma \ref{ext} that $\delta(x^h)=\delta(y)$.
Furthermore, since $x$ and $x^h$ are conjugate in $H$,
we have $\delta(x)=\delta(x^h)$ because $\delta$ is $H$-stable.
This proves that $\delta(x)=\delta(y)$, and hence $\delta$ is $NH$-stable.

By Lemma \ref{ext} again, there is a unique $p$-special character $\hat\delta$ of $N$ that is an extension of $\delta$.
For any element $h\in {\bf N}_H(Q)$, we see that $h$ fixes $\delta$ if and only if
$h$ fixes $\hat\delta$. Since $\delta$ is ${\bf N}_H(Q)$-invariant
and $NH={\bf N}_H(Q)N$, it follows that $\hat\delta$ is $NH$-invariant,
as required.
\end{proof}

We are now ready to prove Theorem A in the introduction,
which we restate here for convenience.

\begin{thm}\label{A}
Let $G$ be a $p$-solvable group for some odd prime $p$,
let $\varphi\in\IBr(G)$ have vertex $Q$, and let $\delta\in\Irr(Q)$ be a linear character.
Suppose that $N$ is a normal subgroup of $G$ such that
$Q\in\Syl_p(N)$ and $\delta$ extends to $N$.
Write $T=N{\bf N}_G(Q,\delta)$. Then the following hold.

{\rm (1)} For any $\chi\in L_\varphi(Q,\delta)$,
each irreducible constituent of $\chi_N$ is factorable.

{\rm (2)} Each Brauer character $\eta\in I_\varphi(T|Q)$
has a unique lift $\tilde\eta\in\Irr(T|Q,\delta)$.
Furthermore, $\tilde\eta^G$ lies in $L_\varphi(Q,\delta)$,
and the map $\eta\mapsto\tilde\eta^G$
defines a bijection $I_\varphi(T|Q)\to L_\varphi(Q,\delta)$.

{\rm (3)} $|L_\varphi(Q,\delta)|\le |{\bf N}_G(Q):{\bf N}_G(Q,\delta)|$.
\end{thm}

\begin{proof}
(1) By definition, there is a nucleus $(W,\gamma)$ for $\chi$,
such that $\gamma\in\Irr(W)$ is factorable that induces $\chi$, $Q\in\Syl_p(W)$ and $(\gamma_p)_Q=\delta$. In particular, $\delta$ extends to $W$,
so it is $W$-stable.
Since $\gamma$ is factorable,
we can write $\gamma=\alpha\beta$, where $\alpha=\gamma_p$ and $\beta=\gamma_{p'}$.

We claim that $\alpha$ extends to a $p$-special character of $NW$.
To see this, observe that $Q$ is also a Sylow $p$-subgroup of $NW$
and $\delta$ is $NW$-stable by Lemma \ref{stab}.
It follows from Lemma \ref{ext} that $\delta$ has a unique $p$-special extension
to $NW$, say $\tilde\delta$.
Furthermore, note that $\tilde\delta$ is linear because $\delta$ is,
so $\tilde\delta_W$ is the unique $p$-special extension of $\delta$ to $W$.
This forces $\tilde\delta_W=\gamma_p=\alpha$ by the uniqueness of such an extension,
as claimed.

Next we need to show that $\beta^{NW}$ is $p'$-special.
Writing $\xi=\gamma^{NW}$, we see that $\xi^G=\gamma^G=\chi$,
and thus $\xi$ is irreducible.
Note that
$$\xi=(\alpha\beta)^{NW}=(\tilde\delta_W\beta)^{NW}
=\tilde\delta\beta^{NW},$$
so $\beta^{NW}$ is also irreducible.
Furthermore, it is clear that $\xi$ is a lift because $\chi$ is,
and we deduce that
$\xi^0=\tilde\delta^0(\beta^{NW})^0=(\beta^{NW})^0$.
This proves that $\beta^{NW}$ is a lift, and since
$\beta^{NW}(1)=|NW:W|\beta(1)$ is a $p'$-number,
it follows by Lemma \ref{I5.2} that $\beta^{NW}$ is $p'$-special.

Now we have proved that $\xi=\tilde\delta\beta^{NW}$ is factorable,
and hence an irreducible constituent of $\xi_N$ is also factorable.
Recall that  $\xi^G=\chi$, so $\chi_N$ has an irreducible constituent
that is factorable. By Clifford's theorem,
we know that each irreducible constituent of $\chi_N$ is factorable.

(2) By Lemma \ref{stab} again, we see that $\delta$ is $T$-stable (with $\N_G(Q,\delta)$ in place of $H$), and hence by Lemma \ref{WJ},
each $\eta\in I_\varphi(T|Q)\subseteq \IBr(T|Q)$ has a unique lift $\tilde\eta\in\Irr(T|Q,\delta)$.
Observe that
$$(\tilde\eta^G)^0=(\tilde\eta^0)^G=\eta^G=\varphi,$$
so $\tilde\eta^G\in L_\varphi(Q,\delta)$.
It follows that $\eta\mapsto\tilde\eta^G$ is a well defined map from $I_\varphi(T|Q)$ to $L_\varphi(Q,\delta)$.

Let $\hat\delta\in\Irr(N)$ be the unique $p$-special extension of $\delta$ to $N$
(see Lemma \ref{ext}). We will prove that for any $\chi\in L_\varphi(Q,\delta)$,
there exists a $p'$-special character $\alpha$ of $N$
such that $\alpha\hat\delta$ is a factorable irreducible constituent of $\chi_N$,
and induction defines a bijection $\Irr(T|\alpha\hat\delta)\to \Irr(G|\alpha\hat\delta)$.

Actually, since $\chi_N$ has an irreducible constituent that is factorable by part (1),
we can choose a nucleus $(W,\gamma)$ for $\chi$
with $N \le W$, and such that $\gamma$ is factorable with $\gamma^G=\chi$, $Q\in\Syl_p(W)$ and $(\gamma_p)_Q=\delta$ (see the construction of
nuclei given at Page 2763 of \cite{N2002}).
Let $\theta$ be an irreducible constituent of $\gamma_N$.
Then $\theta$ is factorable, and thus $\theta_p=\gamma_p$ because $\gamma_p(1)=\delta(1)=1$.
This proves that $\theta_p$ is a $p$-special extension of $\delta$ to $N$,
so $\hat\delta=\theta_p$ by Lemma \ref{ext} again.
Writing $\alpha=\theta_{p'}$, we see that $\theta=\alpha\hat\delta$, which is factorable lying under $\chi$.
Furthermore, it follows from Lemma \ref{stab} that $\hat\delta$ is invariant in
$T=N{\bf N}_G(Q,\delta)$, and by Lemma \ref{prod},
we conclude that the inertial group $I$ of $\alpha\hat\delta$ in $G$ is contained in $T$.
Applying the Clifford correspondence, we deduce that induction defines two bijections:
$$\Irr(I|\alpha\hat\delta)\to\Irr(T|\alpha\hat\delta) \;\text{and}\; \Irr(I|\alpha\hat\delta)\to\Irr(G|\alpha\hat\delta),$$
which implies that induction also defines a bijection
$\Irr(T|\alpha\hat\delta)\to \Irr(G|\alpha\hat\delta)$, as required.

Now we establish the surjectivity of our map.
Let $\chi\in L_\varphi(Q,\delta)$ be as above.
Then there exists a character $\psi\in\Irr(T|\alpha\hat\delta)$ that induces $\chi$.
Observe that $N\NM T$ and $\alpha\hat\delta$ is factorable lying under $\psi$,
so similar reasoning shows that
there exists a normal nucleus $(U,\mu)$ for $\psi$ with $N\le U$,
and such that $\mu$ is factorable and induces $\psi$.
Since $\psi$ is a lift and $p>2$, we know that $\mu_p$ must be a linear character.
It follows that $(\mu_p)_N=\theta_p=\hat\delta$, and hence $(\mu_p)_Q=\hat\delta_Q=\delta$.
Choose a Sylow $p$-subgroup $P$ of $U$ containing $Q$,
and let $\lambda=(\mu_p)_P$.
Then by definition, $(P,\lambda)$ is a vertex for $\psi$ and $\lambda_Q=\delta$.
Now $\psi^G=\chi$, and thus $(P,\lambda)$ is also a vertex for $\chi$.
By the uniqueness, we see that $(P,\lambda)$ and $(Q,\delta)$ are conjugate in $G$.
This shows that $P=Q$ and $\lambda=\delta$,
and hence $(Q,\delta)$ is also a vertex for $\psi$.
Furthermore, note that $(\psi^0)^G=(\psi^G)^0=\chi^0=\varphi$, so $\psi^0\in\IBr(T)$.
It is easy to see that $Q$ is a vertex for $\psi^0$
because $\delta$ is linear (see Theorem 4.6(d) of \cite{C2010} for example),
and hence $\psi^0\in I_\varphi(T|Q)$.
This proves that our map is surjective.

Finally, we establish the injectivity of the map.
Let $\zeta,\eta\in I_\varphi(T|Q)$ induce the same character
$\chi=\tilde\zeta^G=\tilde\eta^G\in L_\varphi(Q,\delta)$,
where $\tilde\zeta$ and $\tilde\eta$ are the lifts in $\Irr(T|Q,\delta)$ of $\zeta$ and $\eta$, respectively.
Reasoning as before (with $T$ in place of $G$),
there exist $p'$-special characters $\alpha_1$ and $\alpha_2$ of $N$
such that $\alpha_1\hat\delta$ and $\alpha_2\hat\delta$ are factorable irreducible constituents of
$\tilde\zeta_N$ and $\tilde\eta_N$, respectively.
Of course both $\alpha_1\hat\delta$ and $\alpha_2\hat\delta$ lie under $\chi$,
and thus they are conjugate in $G$ by Clifford's theorem.
Since $G=N\N_G(Q)$ by the Frattini argument, we can write $\alpha_1\hat\delta=(\alpha_2\hat\delta)^x$
for some $x\in \N_G(Q)$,
which forces $\alpha_1=(\alpha_2)^x$ and $\hat\delta=\hat\delta^x$ by Lemma \ref{prod}.
In particular, the element $x$ fixes $\hat\delta_N=\delta$ and thus $x\in\N_G(Q,\delta)\le T$.
This shows that $\tilde\zeta_N$ and $\tilde\eta_N$ have a
common irreducible constituent $\alpha_1\hat\delta$,
so $\tilde\zeta, \tilde\eta\in\Irr(T|\alpha_1\hat\delta)$.
Similar reasoning implies that induction defines a bijection
$\Irr(T|\alpha_1\hat\delta)\to\Irr(G|\alpha_1\hat\delta)$,
and hence $\tilde\zeta=\tilde\eta$.
Since $\zeta=\tilde\zeta^0=\tilde\eta^0=\eta$,
we conclude that our map is injective.
This proves (2).

(3) Note that $|I_\varphi(T|Q)|\leq|G:T|$ by Lemma \ref{a}.
Also, it is easy to see that
$$|G:T|=|N\N_G(Q):N\N_G(Q,\delta)|=|{\bf N}_G(Q):{\bf N}_G(Q,\delta)|.$$
By part (2), we deduce that $|L_\varphi(Q,\delta)|=|I_\varphi(T|Q)|\leq |{\bf N}_G(Q):{\bf N}_G(Q,\delta)|$.
The proof is now complete.
\end{proof}

\medskip
As an immediate consequence, we have the following, which is Theorem B in the introduction.

\begin{cor}
Let $G$ be a $p$-solvable group with $p>2$, and
suppose that $\varphi\in\IBr(G)$ has vertex $Q$.
If $K\NM G$ is a $p'$-subgroup such that $KQ\NM G$,
then $|L_\varphi(Q,\delta)|\le |{\bf N}_G(Q):{\bf N}_G(Q,\delta)|$ for each $\delta\in\Irr(Q)$.
In particular, we have $|L_\varphi|\le|Q:Q'|$.
\end{cor}

\begin{proof}
We may assume that $L_\varphi(Q,\delta)$ is nonempty.
Then $\delta$ is a linear character of $Q$ by Lemma \ref{CL1.2}.
Writing $N=KQ$, we see that $Q\in\Syl_p(N)$ and $\delta$ extends to $N$.
Now Theorem A applies, and the result follows.
\end{proof}

We mention that the $\pi$-analogue of Theorem A
also holds if we replace the prime $p$ and Brauer characters
by a set $\pi$ of primes and Isaacs' $\pi$-partial characters,
respectively. The proof is similar, so we do not intend to include it here.

\section*{Acknowledgement}
This work was supported by the NSF of China (No. 12171289)
and the NSF of Shanxi Province (No. 20210302124077).


\end{document}